\documentclass[a4paper]{amsart}

\usepackage{amssymb,amsrefs,mathrsfs,lineno}

\newtheorem{theorem}{Theorem}[section]
\newtheorem{lemma}[theorem]{Lemma}
\newtheorem{proposition}[theorem]{Proposition}
\newtheorem{corollary}[theorem]{Corollary}
\newtheorem*{conjecture}{Conjecture}
\newcommand\K{{\mathbb K}}
\newcommand\N{{\mathbb N}}

\begin{document}

\title[On a conjecture of Goodearl]{\boldmath On a conjecture of Goodearl:
  Jacobson radical non-nil algebras of Gelfand-Kirillov dimension $2$}

\author{Agata Smoktunowicz}
\thanks{The research of the first author was supported by Grant No. EPSRC EP/D071674/1.}
\address{\begin{minipage}{0.8\linewidth}
\noindent Agata Smoktunowicz: \\
Maxwell Institute for Mathematical Sciences\\
School of Mathematics, University of Edinburgh,\\
James Clerk Maxwell Building, King's Buildings, Mayfield Road,\\
Edinburgh EH9 3JZ, Scotland, UK\\
\end{minipage}}
\email{A.Smoktunowicz@ed.ac.uk}

\author{Laurent Bartholdi}
\address{\begin{minipage}{0.8\linewidth}
\noindent Laurent Bartholdi:\\
Mathematisches Institut\\
Georg August-Universit\"at zu G\"ottingen\\
Bunsenstra\ss e 3--5\\
D-37073 G\"ottingen\\
Germany\\
\end{minipage}}
\email{laurent.bartholdi@gmail.com}

\date{12 February 2011}

\begin{abstract}
  For an arbitrary countable field, we construct an associative
  algebra that is graded, generated by finitely many degree-$1$
  elements, is Jacobson radical, is not nil, is prime, is not PI, and
  has Gelfand-Kirillov dimension two. This refutes a conjecture
  attributed to Goodearl.
\end{abstract}

\subjclass[2010]{16N40, 16P90}

\keywords{Goodearl conjecture, Nil algebras, the Jacobson radical,
  growth of algebras, Gelfand-Kirillov dimension}

\maketitle


\section{Introduction}
Consider an algebra $R$ over a field $\K$, generated by a
finite-dimensional subspace $V$. The \emph{Gelfand-Kirillov
  dimension}, or \emph{GK-dimension}, of $R$ is the infimal $d$ such
that $\dim(V+V^2+\cdots+V^n)$ grows slower than $n^d$ as
$n\to\infty$. For example, $\K[t_1,\dots,t_d]$ has GK-dimension
$d$. Which constraints does an associative algebra of finite
Gelfand-Kirillov dimension have to obey?  For example, if $R$ is a
group ring, then the group has polynomial growth, so is virtually
nilpotent by Gromov's celebrated theorem~\cite{gromov:nilpotent}, so
$R$ is noetherian. For elementary properties of the Gelfand-Kirillov
dimension, see~\cite{krause-l:gkdim}.

However, various flexible constructions have produced quite exotic
examples of finitely generated associative algebras (\emph{affine
  algebras} in the sequel) of finite
GK-dimension~\cite{bell:examples}, and it has been hoped at least that
algebras of GK-dimension $2$ would enjoy some sort of classification
--- algebras of GK-dimension $<2$ are well understood, and are
essentially polynomials in at most one variable, by Bergman's gap
theorem~\cite{bergman:growth}, and graded domains of GK-dimension $2$
are essentially twisted co\"ordinate rings of projective
curves~\cite{artin-s:gradedquadratic}.

An element $x$ in a ring $R$ is \emph{quasi-regular} if there exists
$y\in R$ with $x+y+xy=0$. This happens, for instance, if $x$ is
nilpotent (take $y=-x+x^2-x^3+\cdots$). Conversely, if $R$ is graded,
then homogeneous quasi-regular elements are nilpotent. The
\emph{Jacobson radical} $J(R)$ of $R$ is the largest ideal all of
whose elements are quasi-regular. A ring is \emph{radical} if it is
equal to its Jacobson radical; note then, in particular, that it may
not contain a unit (in fact, not even a non-trivial idempotent:
$x^2=x,-x+y-xy=0\Rightarrow -x^2+xy-x^2y=-x^2=-x=0$).

A typical result showing the connection between nillity and the
structure of the Jacobson radical is: $R$ is artinian, then $J(R)$ is
nilpotent. The following structural result was expected:
\begin{conjecture}[Goodearl,~\cite{bell:examples}*{Conjecture~3.1}]
  If $R$ is an affine algebra of GK-dimension $2$, then its Jacobson
  radical $J(R)$ is nil.
\end{conjecture}

We disprove this conjecture, by constructing for every countable field
$\K$ an algebra $R$ over $\K$, which is
\begin{itemize}
\item graded by the natural numbers;
\item generated by finitely many degree-$1$ elements;
\item prime;
\item of Gelfand-Kirillov dimension $2$;
\item equal to its Jacobson radical;
\item not PI (i.e.\ does not satisfy a polynomial identity);
\item not nil.
\end{itemize}

Our strategy is to adapt a construction of the first author,
see~\cite{smoktunowicz:jralgebras2}, by showing that it may yield
non-nil algebras.  Some tools are also borrowed from the second
author's paper~\cite{bartholdi:branchalgebras}; however, the
construction given there is not correct, and indeed not not yield a
radical algebra. One of the goals of this paper is therefore to give a
correct solution to the problem raised by Goodearl.

\section{The construction}
We begin by constructing the following algebra $P$; the proof of this
theorem will be split over the next three sections.

\begin{theorem}\label{thm:main}
  Over every countable field $\K$ of characteristic zero, there exists
  a radical algebra $P$, such that the polynomial ring $P[X]$ is not
  radical.

  Moreover, $P$ may be chosen to have Gelfand-Kirillov dimension two,
  be $\N$-graded and generated by two elements of degree one.
\end{theorem}

\noindent We then show that a sufficiently large ring of matrices over
such a $P$ is not nil:
\begin{proposition}\label{prop:nonnil}
  Let $P$ be a radical algebra such that the polynomial ring $P[X]$ is
  not radical. Then there is a natural number $n$ such that the
  algebra $M_n(P)$ of $n$ by $n$ matrices over $P$ is not nil.
\end{proposition}
\begin{proof}
  Suppose that $P$ is radical and that, for every $n\in\N$, the ring
  $M_n(P)$ is nil. Write $R=P[X]$ and $\mathscr I=XR$; we will deduce
  that $R$ is radical. Observe that $M_n(XP)$ is nil for all $n\in\N$,
  and $\mathscr I=XP+(XP)^2+\cdots$; therefore,
  by~\cite{smoktunowicz:jrgradednil}*{Theorem~1.2}, the ring $\mathscr
  I$ is radical. Notice then that $\mathscr I$ is an ideal in $R$, and
  $R/\mathscr I=P$ is radical. Now, if both $\mathscr I$ and
  $R/\mathscr I$ are radical, then so is $R$.
\end{proof}

\begin{lemma}\label{lem:prime}
  Let $R$ be a non-nil ring. Then there exists a quotient $R/\mathscr
  I$ that is non-nil and prime. If $R$ is graded, then $R/\mathscr I$
  may also be taken to be graded.
\end{lemma}
\begin{proof}
  Let $a\in R$ be non-nilpotent. Let $\mathscr I$ be a maximal ideal
  in $R$ subject to being disjoint with $\{a^n\colon
  n=1,2,\dots\}$. Then $R/\mathscr I$ is still not nil. Consider
  ideals $\mathscr P,\mathscr Q\supsetneqq\mathscr I$ with $\mathscr
  P\mathscr Q\subseteq\mathscr I$. By maximality of $\mathscr I$, we
  have $a^n\in\mathscr P$ and $a^m\in\mathscr Q$ for some $m,n\in\N$;
  but then $a^{m+n}\in\mathscr I$, a contradiction. Therefore,
  $R/\mathscr I$ is prime.

  If $R$ is graded, let $\mathscr I$ be a maximal \emph{homogeneous}
  ideal subject to being disjoint with $\{a^n\colon n=1,2,\dots\}$.
  We claim that $\mathscr I$ is a prime ideal in $R$. Suppose the
  contrary; then there are elements $p,q \notin\mathscr I$ such that
  such that $prq\in \mathscr I$ for all $r\in R$. Write
  $p=p_1+\dots+p_d$ and $q=q_1+\dots+q_e$ in homogeneous components,
  and let $p_i$ and $q_j$ denote those summands, for minimal $i,j$,
  that do not belong to $\mathscr I$.

  By assumption, $prq\in\mathscr I$ for all homogeneous $r\in R$ (say
  of degree $k$); so, by considering the component of degree $i+k+j$
  of $prq$, we see that $p_irq_j$ belongs to $\mathscr I$ for all
  homogeneous $r\in R$ (because $\mathscr I$ is graded), whence
  $p_irq_j\in\mathscr I$ for all $r\in R$.

  Let now $\mathscr P$ be the ideal generated by $p_i$ and $\mathscr
  I$; and, similarly, let $\mathscr Q$ be the ideal generated by $q_j$
  and $\mathscr I$. Then, by maximality of $\mathscr I$, we have
  $a^n\in\mathscr P$ and $a^m\in\mathscr Q$ for some $m,n\in\N$; but
  then $a^{m+n}\in\mathscr P\mathscr Q\subseteq\mathscr I$, a
  contradiction. Therefore, $R/\mathscr I$ is prime.
\end{proof}

\noindent Combining these results, we deduce:
\begin{corollary}\label{cor:main}
  Over any countable field $\K$, there exists a non-nil non-PI radical
  prime algebra $R$, of Gelfand-Kirillov dimension two, $\N$-graded
  and generated by finitely many elements of degree one.
\end{corollary}
\begin{proof}
  Let $P$ be as in Theorem~\ref{thm:main}. By
  Proposition~\ref{prop:nonnil}, the ring $R_0=M_n(P)$ is radical and
  non-nil for $n$ large enough. By Lemma~\ref{lem:prime}, some
  quotient $R$ of $R_0$ is radical and prime. Because $P$ is radical,
  its ring of matrices $R_0$ is also radical, and so is its quotient
  $R$. Because $P$ has GK-dimension $\le2$, so do $R_0$ and $R$. If
  $R$ has GK-dimension $<2$, it would have dimension $\le1$ by
  Bergman's gap theorem~\cite{bergman:growth}, so would be finitely
  generated as a module over its centre by~\cite{small-w:dimension1},
  so $R$'s radical would be nilpotent, a contradiction; therefore, $R$
  has GK-dimension exactly $2$.

  Since $P$ is generated by $2$ elements of degree $1$, the rings
  $R_0$ and $R$ are generated by finitely many elements of degree $1$
  (the elementary matrices).

  Finally, $R$ is not PI; indeed, by the Razmyslov-Kemer-Braun
  theorem~\cite{belov-r:pi}*{\S2.5}, if $R$ were PI then its radical
  would be nilpotent.
\end{proof}

\section{Notation and previous results}
Our notation closely matches that of~\cite{smoktunowicz:jralgebras2}.
In what follows, $\K$ is a countable field and $A$ is the free
associative $\K$-algebra in three non-commuting indeterminates
$x,y,z$. The set of monomials in $\{x,y\}$ is denoted by $M$ and, for
$n\geq 0$, the set of monomials of degree $n$ is denoted by $M(n)$. In
particular, $M(0)=\{1\}$ and for $n\geq1$ the elements in $M(n)$ are
of the form $x_1\cdots x_n$ with $x_i\in \{x, y\}$. The
\emph{augmentation ideal} of $A$, consisting of polynomials without
constant term, is denoted by $\bar A$.

The $\K$-subspace of $A$ spanned by $M(n)$ is denoted by $A(n)$, and
elements of $A(n)$ are called {\em homogenous polynomials of degree
  $n$}. More generally, if $S$ is a subset of $A$, then its
homogeneous part $S(n)$ is defined as $S\cap A(n)$.

The {\em degree}, $\deg f$, of $f \in A$, is the least $d \geq 0$ such
that $f \in A(0) + \cdots + A(d)$. Any $f\in A$ can be uniquely
written in the form $f=f_0+f_1+\cdots+f_d$, with $f_i\in A(i)$. The
elements $f_i$ are the {\em homogeneous components} of $f$.  A (right,
left, two-sided) ideal $\mathscr I$ of $A$ is {\em homogeneous} if,
for every $f\in\mathscr I$, all its homogeneous components belong to
$\mathscr I$.

\begin{lemma}[\cite{smoktunowicz:jralgebras2}*{Lemma~6}]\label{lem:6}
  Let $\K$ be a countable field, and let $\bar A$ be as above.
  Then there exists a subset $Z\subset\{5,6,\dots\}$, and an
  enumeration $\{f_i\}_{i\in Z}$ of $\bar A$, such that
  \[i>3^{2deg(f_i) +2}(\deg(f_i)+1)^2\text{ for all }i\in Z.\]
\end{lemma}

\noindent Define the sequence $e(i)=2^{2^{2^{2^i}}}$, and set
\[S=\bigcup_{i\ge5}\{e(i)-i-1,e(i)-i,\dots,e(i)-1\}.\]
\begin{lemma}[\cite{smoktunowicz:jralgebras2}*{Theorem~9}]\label{lem:9}
  Let $Z$ and $\{f_i\}_{i\in Z}$ be as in Lemma~\ref{lem:6}. Fix
  $m\in Z$, and set $w_m=2^{e(m)+2}$. Then there is a two-sided
  ideal $\mathscr P_m\le \bar A$ such that
  \begin{itemize}
  \item the ideal $\mathscr P_m$ is generated by homogeneous elements of
    degrees larger than $10w_m$;
  \item there exists $g_m\in \bar A$ such that
    $f_m-g_m+f_mg_m\in \mathscr P_m$;
  \item there is a linear $\K$-space $F_m\subseteq A(2^{e(m)})$ such
    that $\mathscr P_m\subseteq \sum_{k=0}^{\infty}A(w_mk)F_mA$ and
    $dim_\K(F_m)<m$.
    \end{itemize}
\end{lemma}

\begin{lemma}[\cite{smoktunowicz:jralgebras2}*{Theorem~10}]\label{lem:10}
  Let $Z$ and $F_m$ be as in Lemma~\ref{lem:9}.  There are
  $\K$-linear subspaces $U(2^n)$ and $V(2^n)$ of $A(2^n)$ such
  that, for all $n\in\N$,
  \begin{enumerate}
  \item $\dim_\K V({2^n})=2$ if $n\notin S$;
  \item $\dim_\K V(2^{e(i)-i-1+j})=2^{2^j}$, for all $i\ge5$ and all
    $j\in\{1,\dots,i-1\}$;
  \item $V(2^n)$ is spanned by monomials;
  \item $F_i\subseteq U(2^{e(i)})$ for every $i\in Z$;
  \item $V(2^n)\oplus U(2^n)=A(2^n)$;
  \item $A(2^n)U(2^n)+U(2^n)A(2^n)\subseteq U(2^{n+1})$;
  \item $V(2^{n+1})\subseteq V(2^n)V(2^n)$;
  \item if $n\notin S$ then there are monomials $m_1, m_2\in
    V(2^n)$ such that $V(2^n)=\K m_1+\K m_2$ and
    $m_2A(2^n)\subseteq U(2^{n+1})$.
  \end{enumerate}
\end{lemma}

\section{New results}
Consider the polynomial ring $A[X]$ in an indeterminate $X$. Consider
the elements $(x+Xy)^n$. Write
\[w(n,i)=\sum_{\substack{m\in M(n)\\\deg_{y}m=n-i, \deg_{x}m=i}}m,
\]
and observe that $(x+Xy)^{2^n}=\sum_{i=0}^{2^n}w(2^n,2^n-i)X^i$.  Let
$W(n)$ denote the linear span of all $w(n,i)$ with
$i\in\{0,\dots,n\}$.

We extend the results of the previous section by imposing additional
conditions on the $U(n)$ and $V(n)$ constructed in
Lemma~\ref{lem:10}. Throughout this section, we use the notation
\[T(2^{n+1}) = A(2^n)U(2^n) + U(2^n)A(2^n).\]
\begin{proposition}\label{prop:main}
  There exist subspaces $U(2^n),V(2^n)\subseteq A(2^n)$ satisfying
  all assumptions from Lemma~\ref{lem:10}, with the additional
  property that
  \begin{enumerate}
  \item[(9)] for all $n\in\N$, if $i\in\N$ be such that
    $\{n,n-1,\dots,n-i\}\subset S$, then
    \[\dim_\K(W(2^n)+U(2^n))\ge\dim_\K U(2^n)+2+i;\]
  \item[(10)] $z\in U(2^0)=U(1)$.
  \end{enumerate}
\end{proposition}

\begin{lemma}\label{lem:case2}
  If $\dim_\K(W(2^n)+U(2^n))\ge\dim_\K U(2^n)+2$ and $m_1,m_2\in V(2^n)$
  are linearly independent, then there exists $h\in\{1,2\}$ such that
  \[\dim_\K(W(2^{n+1})+T(2^{n+1})+m_hV(2^n))\ge\dim(T(2^{n+1})+m_hV(2^n))+2.\]
\end{lemma}
\begin{proof}
  Let $i\ge0$ be minimal such that $w(2^n,i)$ does not belong to
  $U(2^n)$, and let $j>i$ be minimal such that $w(2^n,j)$ does not
  belong to $U(2^n)+\K w(2^n,i)$. By the inductive assumption such
  elements can be found.  By permuting $m_1$ and $m_2$ if necessary,
  we may assume that $w(2^n,i)$ is not a multiple of $m_2$, and we
  choose $h=2$. We have
  \[w(2^{n+1},2i)=\sum_{k=-i}^iw(2^n,i+k)w(2^n,i-k),\]
  and either
  \begin{itemize}
  \item[1.] $k=0$,
  \item[or 2.] $k<0$, in which case $w(2^n,i+k)\in U(2^n)$,
  \item[or 3.] $k>0$, in which case $w(2^n,i-k)\in U(2^n)$.
  \end{itemize}
  \noindent Consequently, we get
  \begin{equation}\label{eq:2i}
    w(2^{n+1},2i)\equiv w(2^n,i)w(2^n,i)\mod T(2^{n+1}).
  \end{equation}
  Consider now
  \[w(2^{n+1},i+j)=\sum_{k=-i}^jw(2^n,i+k)w(2^n,j-k);
  \]
  then either
  \begin{itemize}
  \item[1.] $k<0$, in which case $w(2^n,i+k)\in U(2^n)$,
  \item[or 2.] $0<k<j-i$, in which case $w(2^n,i+k)\in U(2^n)+\K
    w(2^n,i)$ and $w(2^n,j-k)\in U(2^n)+\K w(2^n,i)$,
  \item[or 3.] $k=0$ or $k=j-i$,
  \item[or 4.] $k>j-i$, in which case $w(2^n,j-k)\in U(2^n)$.
  \end{itemize}
  \noindent Consequently, we get
  \begin{multline}\label{eq:i+j}
    w(2^{n+1},i+j)\equiv w(2^n,i)w(2^n,j)+w(2^n,j)w(2^n,i)\\\mod T(2^{n+1})+\K w(2^n,i)w(2^n,i).
  \end{multline}

  Recall now that we have
  \[w(2^n,i)\equiv t_{i1}m_1+t_{i2}m_2\mod U(2^n),\qquad
  w(2^n,j)\equiv t_{j1}m_1+t_{j2}m_2\mod U(2^n)
  \]
  for some $t_{i1}, t_{i2}, t_{j1}, t_{j2}\in\K$. Furthermore,
  $t_{i1}\neq 0$, and the vectors $(t_{i1}, t_{i2})$ and $(t_{j1},
  t_{j2})$ are linearly independent over $\K$. Write $Q=T(2^{n+1})+m_2
  V(2^n)$, so that $Q$ contains $m_2m_2$ and $m_2m_1$.

  It follows from~\eqref{eq:2i} that $w(2^{n+1},2i)\equiv
  t_{i1}^2m_1m_1+t_{i1}t_{i2}m_1m_2\mod Q$; and, because
  $t_{i1}\neq0$, we have $w(2^{n+1},2i)\notin Q$.

  Similarly, from~\eqref{eq:i+j} we get $w(2^{n+1},i+j)\equiv
  2t_{i1}t_{j1}m_1m_1+(t_{j1}t_{i2}+t_{i1}t_{j2})m_1m_2\mod Q+\K
  w(2^{n+1},2i)$; and, because the vectors $(t_{i1}, t_{i2})$ and
  $(t_{j1}, t_{j2})$ are linearly independent, so are
  $(2t_{i1}t_{j1},t_{j1}t_{i2}+t_{i1}t_{j2})$ and $(t_{i1}^2,
  t_{i1}t_{i2}) =t_{i1}(t_{i1},t_{i2})$, so we have
  $w(2^{n+1},i+j)\notin Q+\K w(2^{n+1},2i)$.

  We then get $\dim_\K(W(2^{n+1})+Q)\ge\dim_\K Q+2$ as required.
\end{proof}

\begin{lemma}\label{lem:extend}
  $\dim_\K(W(2^{n+1})+T(2^{n+1}))\ge\dim_\K(W(2^n)+T(2^n))+1$.
\end{lemma}
\begin{proof}
  Let there be $k_1,k_2,\ldots,k_j\in\N$ such that
  \[w(2^n,k_1),w(2^n,k_2),\ldots,w(2^n,k_j)
  \]
  are linearly independent modulo $T(2^n)$. We may assume that the
  sequence $(k_1,\ldots,k_j)$ is minimal with this property in the
  lexicographical ordering.  We claim that the elements
  $w(2^{n+1},2k_j)$ and $w(2^{n+1},k_1+k_m)$ for $1\leq m\leq j$ are
  linearly independent modulo $T(2^{n+1})$.  There are $j+1$ such
  elements, as required. As in~\eqref{eq:2i} we observe
  \[w(2^{n+1},2k_1)\equiv w(2^n,k_1)w(2^n,k_1)\mod T(2^{n+1}),
  \]
  and similarly, for each $m\in\{1,\dots,j\}$ we have
  \begin{multline*}
    w(2^{n+1},k_1+k_m)\equiv w(2^n,k_1)w(2^n,k_m)+
    w(2^n,k_m)w(2^n,k_1)\\\mod T(2^{n+1})+\sum_{\substack{1\le p<m\\1\le q<m}}\K
    w(2^n,k_{p})w(2^n,k_{q}).
  \end{multline*}
  Therefore, $w(2^{n+1},k_1+k_m)$ contains the summand
  $w(2^n,k_1)w(2^n,k_m)+ w(2^n,k_m)w(2^n,k_1)$ which no
  $w(2^{n+1},k_1+k_{p})$ with $p<m$ contains.

  \noindent Finally,
  \begin{multline*}
    w(2^{n+1},2k_j)\equiv w(2^n,k_j)w(2^n,k_j)\\\mod
    T(2^{n+1})+\sum_{p=1}^{j-1}w(2^n,k_p)A(2^n)+A(2^n)w(2^n,k_p),
  \end{multline*}
  so $w(2^{n+1},2k_j)$ contains the summand $w(2^n,k_j)w(2^{n+1},k_j)$
  which none of the previous elements contains. It follows that the
  $j+1$ elements we exhibited are linearly independent modulo
  $T(2^{n+1})$.
\end{proof}

\begin{proof}[Proof of Proposition~\ref{prop:main}]
  We adapt the proof of~\cite{smoktunowicz:jralgebras2}*{Theorem~10}
  to show how the additional assumptions may be satisfied. In fact,
  (10) is already part of the construction.

  Recall that the proof
  of~\cite{smoktunowicz:jralgebras2}*{Theorem~10} constructs sets
  $U(2^{n+1})$ and $V(2^{n+1})$ by induction. The following cases are
  considered:
  \begin{itemize}
  \item[1.] $n\in S$ and $n+1\in S$.
  \item[2.] $n\notin S$.
  \item[3.] $n\in S$ and $n+1\notin S$.
  \end{itemize}

  We modify cases~2 and~3, while not changing case~1, which we repeat
  for convenience of the reader:

  \noindent\textbf{Case 1: \boldmath $n\in S$ and $n+1\in S$.} Define
  $U(2^{n+1}) = T(2^{n+1})$ and
  $V(2^{n+1})=V(2^n)V(2^n)$. Conditions~(6,7) certainly hold. If, by
  induction, Conditions~(3,5) hold for $U(2^n)$ and $V(2^n)$, they
  hold for $U(2^{n+1})$ and $V(2^{n+1})$ as well. Moreover, $\dim_\K
  V(2^n)=(\dim_\K V(2^n))^2$, inductively satisfying
  Condition~(2). Finally, Condition~(9) follows directly from
  Lemma~\ref{lem:extend}.

  \noindent\textbf{Case 2: \boldmath $n\notin S$.} We begin as in the
  original argument: $\dim_\K V(2^n)=2$, and is generated by
  monomials, by the inductive hypothesis. Let $m_1, m_2$ be the
  distinct monomials that generate $V(2^n)$. Then $V(2^n)V(2^n)=\K m_1
  m_1 + \K m_1 m_2+\K m_2 m_1 + \K m_2 m_2$. By Lemma~\ref{lem:case2},
  there exists $h\in\{1,2\}$ such that
  \[\dim_\K(W(2^{n+1})+T(2^{n+1})+m_hV(2^n))\ge\dim(T(2^{n+1})+m_hV(2^n))+2.\]
  Permuting $m_1$ and $m_2$ if necessary, we assume $h=2$, and set
  \[U(2^{n+1}) = T(2^{n+1}) + m_2 V(2^n),\qquad V(2^{n+1})=\K m_1
  m_1+\K m_1 m_2.
  \]
  It is clear that Conditions~(1,3,6,7,9) hold, and Condition~(5)
  follows from
  \begin{eqnarray*}
    \lefteqn{A(2^{n+1}) = A(2^n)A(2^n)}\\
    &=&U(2^n)U(2^n) \oplus U(2^n)V(2^n) \oplus V(2^n)U(2^n)
    \oplus m_1 V(2^n) \oplus m_2 V(2^n)\\
    &=&U(2^{n+1}) \oplus V(2^{n+1}).
  \end{eqnarray*}

  \noindent\textbf{Case 3: \boldmath $n\in S$ and $n+1\notin S$.} We
  begin as in the original argument: we have $n=e(i)-1$ for some
  $i>0$.  By the inductive hypothesis, we have
  $\dim_\K(W(2^n)+T(2^n))\ge\dim_\K T(2^n)+i+1$. One more application of
  Lemma~\ref{lem:extend} gives
  \[\dim_\K(W(2^{n+1})+T(2^{n+1}))\ge\dim_\K T(2^{n+1})+i+2.\]

  So as to treat simultaneously the cases $i\in Z$ and $i\not\in Z$,
  we extend Condition~(4) to all $i\in\N$ by taking $F_i=0$ and $s=0$
  if $i\not\in Z$.

  We know that $F_i$ has a basis $\{f_1, \dots ,f_s\}$ for some
  $f_1,\dots , f_s\in A(2^{e(i)}))$ and $s<i$. Write each $f_j$ as
  $f_j=\bar f_j+g_j$ for $\bar f_j \in V(2^n)V(2^n)$ and $g_j\in
  T(2^{n+1})$. Since $V(2^n)V(2^n)\cap T(2^{n+1})=0$, this
  decomposition is unique.

  Since $s<i$, there are elements $w_1, w_2\in W(2^{e(i)})$ such that
  \[(\K w_1+\K w_2)\cap (T(2^{n+1})+\K\bar f_1+\ldots +\K\bar f_s)=0.
  \]
  Let $P$ be a a linear $\K$-subspace of $V(2^n)V(2^n)$ maximal with
  the properties that $(\K w_1+\K w_2)\cap(P+T(2^{n+1}))=0$ and $\bar
  f_j\in P$ for all $j\in\{1,\dots,s\}$.

  Observe that $P$ has codimension $2$ in $V(2^n)V(2^n)$. Since the
  monomials in $V(2^n)V(2^n)$ form a basis, there are two such
  monomials, say $m_1$ and $m_2$, that are linearly independent modulo
  $P$. Define then
  \[V(2^{n+1})=\K m_1+\K m_2,\qquad U(2^{n+1}) = T(2^{n+1}) + P.
  \]
  Conditions~(5,6) are immediately satisfied.  Since each polynomial
  $f_j = g_j + \bar f_j$ belongs to $U(2^{n+1})$, Condition~(4) is
  satisfied as well.

  To end the proof, observe now that $\{w_1,w_2\}$ are linearly
  independent modulo $U(2^{n+1})$, so $\dim_\K(\K w_1+\K
  w_2+U(2^{n+1}))=\dim_\K U(2^{n+1})+2$; this proves~(9).
\end{proof}

\section{Proof of Theorem~\ref{thm:main}}
We present $P$ as a quotient $\bar A/\mathscr E$ for a suitable
ideal $\mathscr E$; we
follow~\cite{smoktunowicz:jralgebras2}*{page~844}. First, $\mathscr E$
is a graded ideal: $\mathscr E=\mathscr E(1)+\mathscr E(2)+\cdots$, so
it suffices to define $\mathscr E(n)$ for all $n\in\N$. By definition,
$\mathscr E(n)$ is the maximal subset of $A(n)$ such that, if $m\in\N$
be such that $2^{m}\leq n<2^{m+1}$, then
\begin{equation}\label{eq:e}
  A(j)\mathscr E(n)A(2^{m+2}-j-n)\subseteq U(2^{m+1})A(2^{m+1})+
  A(2^{m+1})U(2^{m+1})
\end{equation}
for all $j\in\{0,\dots, 2^{m+2}-n\}$; or, more briefly, $(A\mathscr
EA)(2^m)\subseteq T(2^m)$ for all $m\in\N$.

\begin{theorem}\label{thm:polynotrad}
  The subset $\mathscr E$ is an ideal in $\bar A$. Moreover,
  $P:=\bar A/\mathscr E$ is radical, has Gelfand-Kirillov
  dimension two, is $\N$-graded and generated by two degree-$1$
  elements, and $P[X]$ is not radical.
\end{theorem}
\begin{proof}
  By~\cite{smoktunowicz:jralgebras2}*{Theorem~20}, the GK-dimension of
  $P$ is at most $2$; it is in fact exactly $2$, by Bergman's gap
  theorem~\cite{bergman:growth}. Also, $P$ is radical
  by~\cite{smoktunowicz:jralgebras2}*{Theorem~24}. Moreover, $z\in
  U(1)=\mathscr E(1)$, so $P$ is generated by the images of $x$ and
  $y$ in $\bar A/\mathscr E$.

  Recall that $X$ is a free indeterminate commuting with $x$ and
  $y$. Consider $n\ge2$. By Proposition~\ref{prop:main}, not all
  $w(2^n,i)$ belong to $U(2^n)$, so $(x+Xy)^{2^n}\not\in U(2^n)\otimes
  \K[X]$, so $(x+Xy)^{2^{n-2}}\not\in\mathscr E[X]$ by~\eqref{eq:e},
  so $(x+Xy)^{2^{n-2}}\neq0$ in $P[X]$. Since $n$ may be taken
  arbitrarily large, it follows that $x+Xy$ is not nilpotent.

  If $X$ be now declared to have degree $0$, then $P[X]$ is a graded
  ring, and $x+Xy$ is homogeneous and not nilpotent. However, in a
  graded ring, a homogeneous element belongs to the Jacobson radical
  if and only if it is nilpotent; it therefore follows that $P[X]$ is
  not radical.
\end{proof}

\section{Final remarks}
The methods employed here depend crucially on the hypothesis that $\K$
is countable. We don't know if it there are finitely generated radical
algebras of Gelfand-Kirillov dimension two over an uncountable
field. By Amitsur's theorem, such algebras must be nil.

The argument in Theorem~\ref{thm:polynotrad} requires us, in
particular, to construct a ring $P$ such that $P[X]$ is not graded
nil. We do not know if $P$ is nil; if so, this would be an improvement
over~\cite{smoktunowicz:polynomialnil}, in which Smoktunowicz
constructs a nil ring $R$ such that $R[X]$ is not nil.

We note that, over any countable field, nil algebras of
Gelfand-Kirillov dimension at most three were constructed by Lenagan,
Smoktunowicz and Young~\cites{lenagan-s:nillie,lenagan-s-y:nil}.

It remains an open problem whether there exist affine self-similar
algebras satisfying the conditions of Corollary~\ref{cor:main}.

We are also unable to construct an algebra of quadratic growth (i.e.\
whose growth function is bounded by a polynomial of degree two). The
algebras $R$ constructed here do admit an upper bound on their growth
of the form $\dim_\K(R(1)+\cdots+R(n))\le Cn^2\log(n)^3$,
see~\cite{smoktunowicz:jralgebras2}*{Theorem~20}.

We finally refer to Zelmanov's survey~\cite{zelmanov:openpbalgebras}
for a wealth of similar problems.

\end{document}